\documentclass[12pt]{amsart}
\usepackage{amsmath,amscd,amssymb,amsfonts,verbatim}
\usepackage[hypertex, pagebackref=true]{hyperref}
\hypersetup{colorlinks=false}

\def\bfa{\mathbf a}
\def\bfb{\mathbf b}

\newcommand{\bC}{{\mathbb C}}

\newcommand{\bN}{{\mathbb N}}
\newcommand{\bP}{{\mathbb P}}

\newcommand{\cE}{{\mathcal E}}
\newcommand{\cF}{{\mathcal F}}

\newcommand{\cK}{{\mathcal K}}

\newcommand{\cI}{{\mathcal I}}

\newcommand{\cO}{{\mathcal O}}

\newcommand{\rank}{{\rm rank}}

\newcommand{\codim}{\hbox{\rm codim}\,}

\newcommand{\ra}{\rightarrow}

\def\Om{\Omega}
\def\Lam{\Lambda}
\def\lra{\longrightarrow}
\def\ol{\overline}
\def\me{\medskip}
\def\ni{\noindent}
\def\Tor{{\rm Tor}}
\def\coker{{\rm coker}}

\theoremstyle{plain}
\newtheorem{thm}{Theorem}[section]
\newtheorem{cor}[thm]{Corollary}

\newtheorem{lem}[thm]{Lemma}
\newtheorem{prop}[thm]{Proposition}
\newtheorem{que}[thm]{Question}

\theoremstyle{definition}

\newtheorem{rem}[thm]{Remark}
\newtheorem{example}[thm]{Example}

\newtheorem{ex}[thm]{Example}

\title[Complements and higher resonance varieties of hyperplane arrangements]{Complements and higher resonance varieties of hyperplane arrangements}
\author{Nero Budur}
\address{Department of Mathematics,
University of Notre Dame, 255 Hurley Hall, IN 46556, USA} \email{nbudur@nd.edu}

\thanks {This work was partially supported by NSA grant H98230-11-1-0169.}

\begin{document}

\begin{abstract} Hyperplane arrangements form the geometric counterpart of combinatorial objects such as matroids. The shape of the sequence of Betti numbers of the complement of a hyperplane arrangement is of particular interest in combinatorics, where they are known, up to a sign, as  Whitney numbers of the first kind, and appear as the coefficients of chromatic, or characteristic, polynomials. We show that certain combinations, some nonlinear, of these Betti numbers satisfy Schur positivity. At the same time, we study the higher degree resonance varieties of the arrangement. 
We draw some consequences, using homological algebra results and vector bundles techniques, of the fact that all resonance varieties are determinantal. 
\end{abstract}

\maketitle

\section{Introduction}

Hyperplane arrangements form the geometric counterpart of combinatorial objects such as matroids. The shape of the sequence of Betti numbers of the complement of a hyperplane arrangement is of particular interest in combinatorics, where they are known, up to a sign, as Whitney numbers of the first kind, and appear as the coefficients of chromatic, or characteristic, polynomials.  Using this equivalent terminology, lower bounds  have been determined by  Dowling-Wilson \cite{DW} and improved for connected matroids by Brylawski \cite{Br}. Recently, using singularity theory, Huh \cite{H} proved that the sequence of Betti numbers is log concave.  A more general question is what other polynomials in the Betti numbers satisfy positivity. We show that certain combinations, some nonlinear, of Betti numbers satisfy Schur positivity,  see Theorem \ref{thmMain2}.

\me
At the same time, we study the higher degree resonance varieties $R^{\,i}_j$ of the arrangement. These combinatorial invariants were first introduced by Falk \cite{Fa} and are related to the cohomology of local systems on the complement \cite{L-first} and, conjecturally, to the lower central series and Chen ranks of the fundamental group of the complement \cite{Su}. Resonance varieties are also connected with the critical points of master functions \cite{CDFV-1,CDFV-2}, with the Bethe ansatz equations for Gaudin models of complex simple Lie algebras \cite{Var-Bethe}, and are  key objects in the conjectured combinatorial invariance of characteristic varieties and of Milnor fiber cohomology of hyperplane arrangements, e.g. \cite{Li-comb}.  The varieties $R^{1}_j$ have been studied in detail by many people, e.g. Libgober-Yuzvinsky \cite{LY}, Falk-Yuzvinsky \cite{FY}, etc. In contrast, our knowledge of the higher degree $R^{\,i}_j$ is very limited. We  derive from the fact that all $R^{\,i}_j$ are determinantal  some  results complementing the existing ones on resonance varieties.  Thus the main contribution of this note is to point out the usefulness to the, already diverse, hyperplane arrangement theory of some homological algebra and vector bundles results. That $R^{\,i}_1$ admit equations in terms of minors of matrices is also contained in Denham-Schenck \cite{DS}, an unpublished preprint that was brought to our attention by the authors. See also {\it loc. cit.} for an interpretation of $R^{\,i}_1$ in terms of Ext modules and for the role played by the double Ext spectral sequence.

\me
 To present the results, let $D$ be a hyperplane arrangement of degree $d$ in $\bC^n$. We assume that $D$ is central, essential, and indecomposable; see \ref{subsDefs} for definitions. Let $U=\bP^{n-1}-\bP(D)$. We denote by $H^i(U)$ the complex cohomology group $H^i(U,\bC)$.  The cohomology ring of $U$ is a combinatorial invariant of the hyperplane arrangement $D$ \cite{OS}. Let
\begin{align*}
b_i:&=b_i(U)=\dim_\bC H^i(U) \\
\beta_i:&= b_{i}-b_{i-1}+\ldots + (-1)^{i}b_0 .
\end{align*}
Thus $\beta_{n-1}=\chi(U)$ is the Crapo invariant of the matroid of $D$, and $\beta_i$ is the Crapo invariant of a truncation this matroid. Corollary \ref{corBrbds} gives lower bounds on $b_i$ and $\beta_i$ following \cite{DW, Br}. In Proposition \ref{propNewBounds}, although we show bounds in general weaker than the ones of Corollary \ref{corBrbds} as long as $d$ is not too small compared with $n$, we derive them from general algebraic results.

\me
Let $\bP=\bP(H^1(U))$. By \cite{EPY}, we have a linear  locally free resolution
\begin{align}\label{eqF}
0\lra \cO_\bP(-n+1)\otimes H^0(U)  & \mathop{\lra}^{\phi_0} \cO_\bP(-n+2)  \otimes H^1(U)\mathop{\lra}^{\phi_1}\ldots  \\
\notag\ & \ldots\mathop{\lra}^{\phi_{n-2}\ \ \ } \cO_\bP\otimes H^{n-1}(U)\lra \cF\lra 0 \, .
\end{align}
Here $\cF$ is the sheaf version of the {\it singular module} of the arrangement as defined by Eisenbud-Popescu-Yuzvinksy using the BGG-correspondence \cite{EPY}.

\me
An element $v\in H^1(U)$ defines a complex $(H^\bullet (U), v\cup .)$ via the cup product. Define the {\it resonance varieties} of $U$ to be
$$
R_j^{\,i}(U):=\{ v\in H^1(U)\mid \dim H^i(H^\bullet (U),v\cup .)\ge j \}.
$$
The resonance varieties of $U$ are combinatorial invariants of the hyperplane arrangement $D$. We will use the notation $R^{i}(U)$ for $R^{\,i}_1(U).$ We have $R^{i}(U)=R^{\,i}_1(U)\supset R^{\,i}_2(U)\supset R^{\,i}_3(U)\supset\ldots $. 
It is known that $R^i(U)$ are union of vector subspaces of $H^1(U)$, see \cite{CS, CO, EPY}. 

\me  Define the {\it Fitting ideal} $\cI_{k}(\phi_i)$  to be the ideal generated by the $k$-minors of the matrix of linear forms representing $\phi_i$ in (\ref{eqF}). Matei-Suciu \cite{MS} have shown that $R^1_j(U)$ admit equations in terms of minors of the linearized Alexander matrix, which is the same as $\phi_1$. This can be generalized, see also Denham-Schenck \cite{DS}-Proposition 2.9 for $R^{\,i}_1(U)$.

\begin{thm}\label{propRefRes}
$\bP(R^{\,i}_j(U))$ is the support of the  ideal $\cI_{\beta_i+1-j}(\phi_i)$.
\end{thm}

\begin{cor}\label{corHigherRCodim}
$\text{codim}\; R^{\,i}_j(U)\le \min\{d-1, (\beta_{i-1}+j)(\beta_{i+1}+j)\}.$
\end{cor}

When the above inequality on the codimension of $R^{\,i}_j(U)$ is useful, we can say something stronger about $R^{\,i}_j(U)$. Note that in contrast with the conclusion of the next result, it is known that the irreducible components of $\bP(R^{\;1}_j(U))$ are mutually disjoint \cite{LY}, see also \cite{DPS}.

\begin{cor}\label{corConnected}
If $(\beta_{i-1}+j)(\beta_{i+1}+j)<d-2$, then $\bP(R^{\,i}_j(U))$ is connected.
\end{cor}

It is known that  resonance propagates, that is, $R^i(U)\subset R^{i+1}(U)$, \cite{EPY}. We show that a deeper propagation holds:

\begin{cor}\label{corHigherResInclusions} We have:
\begin{itemize}
\item $R^{\,i}(U)\subset R^{\; i+1}_2(U)$  ($i\le n-2$),
\item $R^{\,i}(U)\subset R^{\; i+1}_j(U)$  ($i<n-2$ and $j\le 1+\frac{n-3}{i+1}$).
\end{itemize}
\end{cor}

\me  Corollary \ref{corHigherRCodim} can be improved when $j=1$. Define from now
$$
q_i:=\codim R^{i}(U) = d-1-\dim R^i (U)
$$
for $0\le i \le n-1$.

\begin{thm}\label{propCod}
$$n-1-i\le q_i\le \min \{d-1, (\beta_{i-1}+1)(\beta_{i+1}+1),\beta_{i+1}+i+1\} .$$
\end{thm}

The computational complexity of the resonance varieties is discussed briefly in section \ref{subsecComp}, see for example Proposition \ref{propFactor}.

\me
Regarding the shape of the Betti numbers of the complement, using a vector bundle method of Popa-Lazarsfeld \cite{LP,Lo} we obtain Schur positivity of certain combinations of the Betti numbers of $U$. Define for $j>0$
$$c^{(j)}_t:= \prod_{k=1}^{j+1} (1-kt)^{(-1)^kb_{j+1-k}}.$$ Let $c_i^{(j)}$ denote the coefficient of $t^i$ in $c^{(j)}_t$. 

\begin{thm}\label{thmMain2} Let $0<j<n-1$. If $j=n-2$, assume that $q_{n-2}>1$. Then:

\ni (a) Any Schur polynomial of weight $<q_j$ in $c_1^{(j)},\ldots ,c_{q_j-1}^{(j)}$  is non-negative. In particular, $c_i^{(j)}\ge 0$ for $1\le i <q_j$, and
\begin{equation}\label{eqcj}
c_1^{(j)}=\sum_{k=1}^{j+1}(-1)^{k+1}\cdot k\cdot b_{j+1-k} \ge 0.
\end{equation}

\ni (b) $c_i^{(j)}=0$ for $i>\min \{\beta_{j+1},q_j-1\}$.

\ni (c) $q_j> \max\{ i\mid c^{(j)}_i\ne 0 \}$.

\ni (d) The coefficients $c^{(j)}_i$ of the polynomial $c^{(j)}_t$ form a log concave sequence.
\end{thm}
\ni Recall that the first few Schur polynomials are: $c_1$; $c_2$, $c_1^2-c_2$; $c_3$, $c_1c_2-c_3$, $c_1^3-2c_1c_2+c_3$. Other lower bounds on $b_i$, in terms of $b_1,\ldots ,b_{i-1}$ can determined from Theorem \ref{thmMain2}, see \ref{subsLower}.

\me The outline of the article is the following. In the second section we recall the basic definitions and lower bounds on Betti numbers of $U$ from \cite{DW, Br}. The third section is the core of the article, where we prove the statements from this Introduction. Next section is a brief discussion of Theorem \ref{thmMain2}. We end the article with a section containing some remarks about resonance varieties.

\me I would like to thank A. Dimca, M. Falk, A. Suciu, and S. Yuzvinsky  for comments and suggestions, to G. Denham and H. Schenck for sharing their preprint \cite{DS}, and to the Max Planck Institute in Bonn and the Johns Hopkins University for their hospitality during the writing of the article.

\section{Preliminaries}

\subsection{Hyperplane arrangements.}\label{subsDefs} Let us recall the basic terminology. An {\it affine} (resp. {\it projective}) {\it hyperplane arrangement} in $\bC^n$ (resp. $\bP^{n-1}$) is a finite set of hyperplanes. We will abuse notation and identify the hyperplane arrangement with the corresponding reduced divisor. An arrangement is  {\it essential} if the intersection of all hyperplanes  has dimension at most zero. An arrangement is {\it central} if the intersection of all hyperplanes is nonempty. Non-essential implies central. An arrangement is {\it indecomposable} if it is not the product of two distinct hyperplane arrangements. In other words, there is no choice of coordinates for which the equation of the arrangement is a product of two non-constant polynomials in two disjoint sets of variables.
 
\me
An affine central arrangement $D$ will tacitly be assumed to contain the origin in any of its hyperplanes.  For a hyperplane arrangement $D$ and a linear subspace $S$ of the ambient space, we will denote by $D_{|S}$ the hyperplane arrangement $D\cap S$ in $S$.

\me
For every affine hyperplane arrangement $D$ in $\bC^n$ we will consider, initially, the following sets of numbers: $h_i$, $b_i$, $\beta_i$. These are as follows:
\begin{align*}
h_i &=\dim H^i(\bC^n-D,\bC),\\
b_i&=h_i-h_{i-1}+\ldots +(-1)^{i}h_0,\\
\beta_i  &= b_{i}-b_{i-1}+\ldots + (-1)^{i}b_0\\
& = h_i-2h_{i-1}+3h_{i-2}-\ldots +(-1)^i h_0.
\end{align*}
When $D$ is central, the numbers $h_i$ are also known as the absolute values of the {\it Whitney numbers of the first kind}, and the number $\beta_{n-1}$ is commonly called  the  {\it Crapo invariant}.

\subsection{Central affine arrangements.}
Define
$$
P^{DW}(d,n,i):=\binom{n}{i} +(d-n)\binom{n-1}{i-1},
$$
and
$$
P^{B}(d,n,i):= \binom{n}{i} +(d-n)\binom{n}{i-1}  -\delta_{i,n-1}.
$$

\begin{thm}{\bf (Dowling-Wilson \cite{DW})} Let $D$ be a central essential hyperplane arrangement in $\bC^n$ of degree $d$. Then
\begin{equation*}\label{eqDW}
h_i \ge P^{DW}(d,n,i).
\end{equation*}
\end{thm}

In the indecomposable case, we have the following improvement. Let 
$$\Om:=\{(x,y)\in \bN^2\mid  x-y\ge 2, (x,y)\ne (7,4)\}.$$

\begin{thm}\label{thmBr}{\bf (Brylawski \cite{Br})} Let $D$ be an indecomposable central essential hyperplane arrangement in $\bC^n$ of degree $d$. If $(d,n)\in \Om$ then 
\begin{equation*}\label{eqBr}
h_i\ge P^B(d,n,i)
\end{equation*}
if $i<n$, and
\begin{equation*}\label{eqBr2}
\beta_{n-1}\ge \max \{1, d+2-2n\}.
\end{equation*}
\end{thm}
\ni For examples when the bounds are achieved see {\it loc. cit.}

\me
Let $D$ be a central hyperplane arrangement in $\bC^n$. Then
$$
b_i=H^i(\bP^{n-1}-\bP(D),\bC).
$$
Indeed,
$$
\pi (D,t)=(1+t)\pi (\bP(D),t),
$$
where we denote by $\pi(Z,t)$ the Poincar\'e polynomial of the complement of $Z$, \cite{OT}. Note:
\begin{align*}
h_i &=b_i+b_{i-1},\\
\beta_i & = b_{i}-b_{i-1}+\ldots + (-1)^{i}b_0,\\
b_i&=h_i-h_{i-1}+\ldots +(-1)^{i}h_0 = \beta_i +\beta_{i-1}.
\end{align*}

\me Let $S$ be a generic subvector space of $\bC^n$ of dimension $s+1$. Then
$$
b_i(D)=b_i(D_{|S})
$$ 
for $0\le i\le s$, by the combinatorial invariance of $b_i$, \cite{OS}.  Hence also
\begin{align*}
h_i(D)&=h_i(D_{|S}),\\
\beta_i(D)&=\beta_i(D_{|S}),
\end{align*}
for $0\le i\le s$. If $D$ is indecomposable, then so is $D_{|S}$ by Corollary \ref{corBeta}. If $D$ is essential, so is $D_{|S}$. Applying Theorem \ref{thmBr}, we obtain:

\begin{cor}\label{corBrbds}
Let $D$ be a central essential indecomposable hyperplane arrangement in $\bC^n$ of degree $d$. For $0< i< n$,
$$
h_i\ge \max \{ P^B(d,s+1,i)\mid i\le s\le n-1, (d,s+1)\in \Om \};
$$
if $(d,i+1)\in\Om$, then
$$
\beta_i\ge\max\{1, d-2i\};
$$
and if in addition $(d,i)\in\Om$, then
$$
b_i\ge 2(d-2i-1).
$$
\end{cor}

\begin{rem} Lower bounds on $h_i=b_i+b_{i-1}$ do not automatically translate into lower bounds for $b_i$, unless the following is true: the Betti numbers of the complement of an affine arrangement of degree $d-1$ in $\bC^{n-1}$ are the first $n-1$ Betti numbers of the complement of a central affine arrangement of degree $d-1$ in some $\bC^{n'}$ with $n'\ge n$. Here $n'=n$ iff the hyperplane at infinity is in general position.  
\end{rem}

\section{Proofs}

\subsection{Lower bounds.} Let us point out how homological algebra results about linear free resolutions imply lower bounds on $h_i$, $b_i$, $\beta_i$. These bounds are in general weaker than the lower bounds of Corollary \ref{corBrbds},  as long as $d$ is not too small compared with $n$. 

\me
We start with a different proof of the following result. Note that it is well-known that the positivity of the Crapo invariant $\beta_{n-1}$ is equivalent to the indecomposability of the central arrangement.

\begin{prop}\label{corBeta} Let $D$ be a central essential indecomposable hyperplane arrangement in $\bC^n$. Then
$\beta_i  > 0 $ for $0\le i <n$.
\end{prop}
\begin{proof} We can assume $0<i<n$. Let $\cF$ and $\phi_i$ be as in (\ref{eqF}). Then $\beta_{i}=\rank\; \phi_i $. So $\beta_i\ge 0$ since the rank is a nonnegative number. If $\beta_i=0$ then $\phi_i=0$. This is a contradiction. Indeed, by the definition of $\phi_i$ in \cite{EPY}, the condition that $\phi_i=0$ implies that $H^{k}(U)=0$ for $k>i$. In particular, $\beta_{n-1}=0$ which contradicts the indecomposability of the arrangement. Note that since the entries of the matrix representing $\phi_i$ are linear forms and not all vanishing, a Nakayama Lemma argument implies that (\ref{eqF}) is the minimal locally free resolution of $\cF$, see \cite{E}-Lemma 19.4.
\end{proof}

\begin{prop}\label{propNewBounds} $\ $

(a) $b_i\ge \binom{n-1}{i}$, hence $h_i\ge\binom{n}{i}$.

(b) $\beta_i\ge n-1-i$ for $i>0$.
\end{prop}
\begin{proof}
By the proof of Proposition \ref{corBeta}, the projective dimension of $\cF$ is $\text{pd}(\cF)=n-1$. A conjecture of Buchsbaum-Eisenbud and Horrocks about Betti numbers of minimal free resolutions implies that $b_i\ge \binom{\text{pd} (\cF)}{i}$. This conjecture is proved for graded modules with a linear minimal free resolution, which is our case, by Herzog-K\"{u}hl \cite{HK}. This is part (a).
Part (b) is the lower bound for syzygy modules of a module with projective dimension $n-1$ due to Evans-Griffith \cite{EG}.
\end{proof}

\subsection{Resonance, singular modules, and truncations.} Let $D$ be a central indecomposable hyperplane arrangement in $\bC^n$. Consider the locally free resolution (\ref{eqF}) of the singular module $\cF$ of $D$. Denote by $\cK_i$ the cokernel of $\phi_i$:
$$
\cO_{\bP} (-n+1+i) \otimes  H^{i}(U)\mathop{\lra}^{\phi_i} \cO_{\bP}(-n+2+i)\otimes H^{i+1}(U)\lra \cK_i\lra 0 .
$$

\begin{prop}\label{propVary} The complex (\ref{eqF}), without the last term $\cF$, is the complex of sheaves  on $\bP$ obtained from the complexes of vector spaces $(H^\bullet (U), v\cup .)$ when $v$ varies along $\bP$.
\end{prop}
\begin{proof} Let $S=\bC[x_1,\ldots ,x_{b_1}]$. The  complex of sheaves on $\bP$  obtained by varying $v\in \bP$ comes from the complex of free $S$-modules with maps
$$
S\otimes H^i(U)\lra S\otimes H^{i+1}(U)
$$
given by $1\otimes w\mapsto \sum x_i\otimes ( e_i\cup w)$, where $e_i$ is a basis of $H^1(U)$ and $x_i$ is the dual basis of $H^1(U)^\vee$. These are exactly the maps of the complex of $S$-modules corresponding to (\ref{eqF}) by \cite{EPY}-3.
\end{proof}

\me
Let $S$ be a generic subvector space of $\bC^n$ of dimension $s+1$. By Proposition \ref{corBeta}, the hyperplane arrangement $D_{|S}$ is also indecomposable. Since $H^i(U)=H^i(U\cap \bP(S))$ for $0\le i\le s$, the resonance varieties also agree up to degree $s$:
$$
R^i(U)=R^i(U\cap\bP(S)),\quad\quad\text{ for }0\le i\le s.
$$

The singular modules of $D$ and $D_{|S}$ are related as follows.  To denote the dependance on $D$, we will briefly use the notation $\cF_{D}$ for $\cF$.

 \begin{lem}\label{lemResTrunc} The singular module of $D_{|S}$ is
$
\cF_{D_{|S}}=\cO_{\bP}(n-1-s)\otimes \cK_{s-1}.
$
\end{lem}
\begin{proof} The locally free resolution (\ref{eqF}) of $\cF_D$ is minimal, as in the proof of Proposition \ref{corBeta}.
Thus the complex
$$
0\ra\cO_\bP (-s)\otimes H^0(U)\ra\ldots\ra\cO_{\bP}\otimes H^s(U) ,
$$
formed from the truncation of the complex (\ref{eqF}) after twisting by $\cO_\bP (n-1-s)$, is the minimal locally free resolution of $\cO_{\bP}(n-1-s)\otimes \cK_{s-1}$. However, by Proposition \ref{propVary}, the minimal locally free resolution of $\cF_{D_{|S}}$ has the same shape due to the invariance of the cup product maps up to degree $s$. Hence $
\cF_{D_{|S}}=\cO_{\bP}(n-1-s)\otimes \cK_{s-1}.
$
\end{proof}

\subsection{The resonance varieties $R^i(U)$.} Before discussing the refined resonance varieties $R^{\; i}_j(U)$, we focus on the case $j=1$. Let $\cI(\phi)$ denote the ideal $\cI_{\beta_i}(\phi_i)$ given by the $\beta_i$-minors of $\phi_i$. We prove first a particular case of Theorem \ref{propRefRes}.

\begin{prop}\label{corFitt}
$\bP(R^{i}(U))$ is the support of the Fitting ideal $\cI(\phi_i)$.
\end{prop}
\begin{proof} Let $\cK_i:={\rm{coker}} \;\phi_i$.  We need to show that $\bP(R^{i}(U))$ is the locus of points where $\cK_i$ fails to be locally free. Indeed, by \cite{E}-20.6, the non-locally free locus of $\cK_i$ is the support of the Fitting ideal $\cI(\phi_i)$. 

 Let $v\in \bP=\bP(H^1(U))$ and let $\kappa(v)$ be the residue field of $v$. By \cite{EPY}-Theorem 4.1 (a), 
$$
H^{n-1-i}(H^\bullet(U), v\cup .)=\text{Tor} _i^{\cO_{\bP,v}}(\kappa (v),\cF_v)
$$
In particular, we have $H^{n-2}(H^\bullet(U), v\cup .)=\text{Tor} _1^{\cO_{\bP,v}}(\kappa (v),\cF_v)$. By \cite{E}-Ex. 6.2 (a), the space on the right is zero iff $\cF_v$ is free. Hence $v\in\bP (R^{n-2}(U))$ iff $\cK_{n-2}=\cF_v$ is not free. 

 To prove the claim for $i<n-2$, we reduce to the above case by truncating and using Lemma \ref{lemResTrunc}. In this case we have
\begin{align}\label{eqTruncTor}
H^i(H^\bullet (U),v\cup .)=\text{Tor}_1^{\cO_{\bP,v}}(\kappa (v),(\cK_i)_v).
\end{align}
Hence, as above, $\bP(R^{i}(U))$ is the locus of points where $\cK_i$ fails to be locally free. 
\end{proof}

\begin{rem}
Note that propagation of resonance, i.e. the fact that $R^{i}(U)\subset R^{i+1}(U)$, follows from \cite{E}-20.12 where it is shown that the support of $\cI(\phi_i)$ is included the support of $\cI(\phi_{i+1})$. 
\end{rem}

\me\ni{\it Proof of Theorem \ref{propCod}.}  Let us prove the first inequality. By \cite{E}-20.9, we have that ${\rm depth}\; \cI(\phi_i)\ge n-1-i.$ Since $S=\bC[x_1,\ldots ,x_{b_1}]$ is Cohen-Macaulay, depth equals codimension. Now the claim follows by Proposition \ref{corFitt}.  

It is known that the depth of Fitting ideals is  bounded above by that of generic determinantal varieties. More precisely, for a map $f:P\ra Q$ of projective modules,  we have by \cite{EN} that
\begin{align}\label{eqEagon}
{\rm depth}\, I_k(f) \le ({\rm rank} (P)-k+1)({\rm rank} (Q) -k+1),
\end{align}
where $I_k(f)$ is the ideal generated by the $k$-minors of the matrix representing $f$. In our case,
\begin{align*}
q_i={\rm depth}\, \cI(\phi_i) & \le (b_i-\beta_{i}+1)(b_{i+1}-\beta_{i}+1)=\\
& =(\beta_{i-1}+1)(\beta_{i+1}+1).\quad\quad
\end{align*}
The remaining inequality  $q_i\le \beta_{i+1}+i-1$ is part (d) of Theorem \ref{thmMain2-v2} below.  $\quad\Box$

\me
Theorem \ref{thmMain2}  consists of the parts (a)-(c) and (e) in:

\begin{thm}\label{thmMain2-v2} Let $0<j<n-1$. If $j=n-2$, assume that $q_{n-2}>1$. Then:

\ni (a) Any Schur polynomial of weight $<q_j$ in $c_1^{(j)},\ldots ,c_{q_j-1}^{(j)}$  is non-negative. 

\ni (b) $c_i^{(j)}=0$ if $\beta_{j+1}<i<q_j$.

\ni (c) $q_j> \max\{ i\mid c^{(j)}_i\ne 0 \}$.

\ni (d) $\beta_{j+1}\ge q_j-j-1$ .

\ni (e) The coefficients $c^{(j)}_i$ of the polynomial $c^{(j)}_t$ form a log concave sequence.

\end{thm}
\begin{proof}
The proof is a combination of \cite{EPY} and \cite{LP}. Let us prove first the case when $j=n-2$. Denote $q_{n-2}$ by $q$, and $\beta_{n-1}$ by $\beta$.  

Let $W$ be a vector subspace of $H^1(U)$ that is transversal to $R^{n-2}_1(U)$. Then $\bP':=\bP(W)$ has dimension $q-1\ge 1$. Restrict the linear locally free resolution of $\cF$ to $\bP'$. Then we have a linear locally free resolution
\begin{align}\label{eqSec}
0\ra \cO_{\bP'}(-n+1)\otimes H^0(U) & \ra \cO_{\bP'}(-n+2)\otimes H^1(U)\ra\ldots \\
\notag \ & \ldots \ra \cO_{\bP'}\otimes H^{n-1}(U)\ra \cF'\ra 0 \, ,
\end{align}
where $\cF'$ is a vector bundle on $\bP'$.

It follows that $c_t^{(n-2)}$ is the Chern polynomial of $\cF'$. Since $\cF'$ is globally generated, the Chern classes $c_i(\cF')$, which equal $c_i^{(n-2)}$, and the Schur polynomials in these, are nonnegative, \cite{Fu}-12.1.7-(a). This proves (a).

As in \cite{LP}, the parts (b) and (c) follow from the fact that  $\rank (\cF')=\beta$, that $c_i(\cF')=0$ for $i>\max \{\rank (\cF'),q-1\}$, and that there exist an $i$ such that $c_i(\cF')\ne 0$.

The proof of part (d) is essentially the same as the one in \cite{LP}. Since $\beta>0$, we can assume $q>n-1$. If $q=n$ then we need to show that $\cF'\ne 0$. If $\cF'=0$ then (\ref{eqSec}) cannot be an exact sequence, as the alternating product of Chern polynomials cannot be 1. So we can assume that $q>n$. Chasing through (\ref{eqSec}) we have that $H^j(\bP',\cF'(k))=0$ for all $k$ and $0<j<q-n+1$. The splitting criterion of Evans-Griffith, see \cite{La}-3.2.12, implies, if $\rank (\cF')\le q-n+1$, that $\cF'$ splits as a direct sum of line bundles. This cannot happen for the same reason as before. Hence $\rank (\cF')> q-n+1$.

Now we prove the case $j<n-2$. Consider the complex obtained from (\ref{eqF}) by truncation:
\begin{align*}
0 & \lra \cO_\bP(-n+1)^{\oplus b_0}   \lra \cO_\bP(-n+2)^{\oplus b_1}\lra\ldots  \\
\notag\ & \ldots \lra \cO_\bP (-n+1+j)^{\oplus b_{j}} \lra \cO_\bP(-n+j+2)^{\oplus b_{j+1}}\lra \cK_j\lra 0 \, .
\end{align*}
We tensor this complex with $\cO_\bP(n-j-2)$ to obtain a linear locally free resolution 
\begin{align*}
0 &  \lra \cO_\bP(-1-j)^{\oplus b_0}    \lra \cO_\bP(-j)^{\oplus b_1}\lra\ldots  \\
\notag\ & \ldots \lra \cO_\bP (-1)^{\oplus b_{j}} \lra \cO_\bP^{\oplus b_{j+1}}\lra \cK_j(n-j-2)\lra 0 \, .
\end{align*}
Now $\cK_j(n-j-2)$ is globally generated, and the rest of the proof goes as for the case $j=n-2$.

By part (c), $c^{(j)}_t$ is a polynomial which has all coefficients and all roots real. An old theorem of Newton, see \cite{Sta}-Theorem 2, implies that the coefficients must form a log concave sequence. This shows part (e).
\end{proof}

\begin{que}\label{queCod}
Is $q_{n-2}>1$ always true ?
\end{que}
We note that this is true for $n=3$: the non-local components of $R^1(U)$ have small dimension \cite{Y}, and the local components have codimension $\ge 2$ by the indecomposability of $D$, cf. \cite{FY}. We give an algebraic reformulation of this question and a partial answer in section \ref{subsRefQ} .

\subsection{The refined resonance varieties $R^{\,i}_j(U)$.} We prove now the remaining statements from Introduction.

\me\ni{\it Proof of Theorem \ref{propRefRes}.}
By  (\ref{eqTruncTor}) and Lemma \ref{lemResTrunc}, it is enough to prove the case $i=n-2$. Hence, we need to show that the locus of points $v$ such that $\dim \text{Tor}_1^{\cO_{\bP,v}}(\kappa(v),\cF_v)\ge j$ is the support of $\cI_{\beta_{n-2}+1-j}(\phi_{n-2})$. This ideal is actually an invariant of $\cF$, by \cite{E}-Corollary/Definition 20.4. In the notation of {\it loc. cit.},
$$
\cI_{\beta_{n-2}+1-j}(\phi_{n-2})=\text{Fitt}_{\beta_{n-1}-1+j}(\cF),
$$
since $\rank\, \cF=\beta_{n-1}$. By \cite{E}-Proposition 20.6, the support of the Fitting ideal $\text{Fitt}_{\beta_{n-1}-1+j}(\cF)$ is the locus of points $v$ where $\cF_v$ cannot be generated by $\beta_{n-1}-1+j$ elements.

To summarize, it is enough to show: if $(R,P)$ is a local noetherian domain, and $M$ is an $R$-module minimally generated by $k$ elements, then 
$$
k-\rank (M)\le \dim \text{Tor}_1^R(R/P,M).
$$
Here $\rank (M)$ is the dimension of the $K$-vector space $M\otimes_R K$, where $K$ is the quotient field of $R$.

To prove this claim, consider a minimal set of generators of $M$ and the short exact sequence attached to them:
\begin{align}\label{eqShS}
0\ra N\ra R^k\ra M\ra 0 .
\end{align}
Since the operation $.\otimes_R K$ is exact, we have
$$
k-\rank(M)=\rank (N).
$$
By a similar reasoning, $\rank (N)\le k'$, where $k'$ is the minimal number of generators of $N$. Tensoring (\ref{eqShS}) with $R/P$ we have  an exact sequence of $R/P$-vector spaces
$$
0\ra \text{Tor}_1^R(R/P,M)\ra N\otimes_R R/P\ra(R/P)^k\ra M\otimes_R R/P\ra 0 .$$
By Nakayama Lemma and the minimality of $k$, the two vector spaces on the right are isomorphic. Hence, also the two vector spaces on the left are isomorphic, and $k'=\dim \text{Tor}_1^R(R/P,M)$.
$\quad\Box$

\me\ni
{\it Proof of Corollary \ref{corHigherRCodim}.}
It follows from Theorem \ref{propRefRes} and (\ref{eqEagon}).$\quad\Box$

\me\ni
{\it Proof of Corollary \ref{corConnected}.}
By a theorem of Fulton-Lazarsfeld \cite{La} -7.2, the support of the ideal sheaf $\cI_{k+1}(\phi)$ is connected if $\phi:\cE\ra\cE'$ is a vector bundle map on a variety $X$ such that $\cE^\vee\otimes\cE'$ is ample and $\dim (X)>(\rank (\cE)-k)(\rank (\cE')-k)$. We apply this to $\cI_{\beta_i+1-j}(\phi_i)$ and $X=\bP^{d-2}$. Note that 
$$
(\cO_\bP(-n+1+i)^{\oplus b_i})^\vee\otimes\cO(-n+2+i)^{\oplus b_{i+1}}=\cO_\bP(1)^{\oplus b_ib_{i+1}}
$$
is an ample vector bundle since it is the direct sum of ample line bundles.
$\quad\Box$

\me\ni
{\it Proof of Corollary \ref{corHigherResInclusions}.}
 Because of Theorem \ref{propRefRes}, we can apply to the refined resonance varieties facts about minors of matrices in a finite free resolution. For example, the second claim follows directly from Buchsbaum-Eisenbud \cite{BE}- Corollary 6.2 together with \cite{BE}-Theorem 3.1, part (b). The first claim follows from \cite{BE}-(10.5) if the Conjecture 10.1 in \cite{BE} holds. This conjecture is proven by Tchernev-Weyman \cite{TW}.
$\quad\Box$

\section{Dissection of Theorem \ref{thmMain2}}\label{secRem}

\subsection{Linear combinations.}  We note that
\begin{align*}
c^{(i)}_1 &=\beta_i-\beta_{i-1}+\beta_{i-2}-\ldots\\
& = b_i-2b_{i-1}+3b_{i-2}-\ldots\\
&=h_i-(1+2)h_{i-1}+(1+2+3)h_{i-2}-(1+2+3+4)h_{i-3}+\ldots.
\end{align*}
Since $\beta_i=c^{(i)}_1+c^{(i-1)}_1$, lower bounds on the numbers $c^{(i)}_1$ give lower bounds on $\beta_i$.

\subsection{Nonlinear combinations.} The numbers $c^{(i)}_j$ for $i\ge 2$, and the higher degree Schur polynomials, are nonlinear combinations of $h_i$, $b_i$, $\beta_i$, $c^{(i)}_1$. Denote for simplicity
$$
a_i=c^{(i)}_1.
$$
We have for example:
\begin{align*}
c^{(2)}_2\ge 0  & \iff a_2^2+2a_0-2a_1+a_2\ge 0 \\
[c^{(2)}_1]^2-c_2^{(2)} \ge 0& \iff  a_2^2-2a_0+2a_1-a_2\ge 0\\
c^{(3)}_2\ge 0 &\iff a_3^2-2a_0+2a_1-2a_2+a_3\ge 0\\
c^{(3)}_3\ge 0 &\iff a_3^2 +6(a_1a_3-a_0a_3+a_2a_3)+\\
	&\quad\quad +3a_3^2-36a_0+24a_1-12a_2+2a_3\ge 0 \\
c^{(3)}_1c^{(3)}_2-c^{(3)}_3\ge 0 & \iff a_3^3+18a_0-12a_1+6a_2-a_3\ge 0\\
[c^{(3)}_1]^3-2c^{(3)}_1c^{(3)}_2+c^{(3)}_3\ge 0 &\iff a_3^3+6(a_0a_3-a_1a_3+a_2a_3)-3a_3^2-\\
&\quad\quad-36a_0+24a_1-12a_2+2a_3\ge 0.	
\end{align*}

\subsection{Lower bounds on $b_i$.}\label{subsLower} Although we regard the positivity properties from Theorem \ref{thmMain2} on various combinations of the numbers $b_i$ as mainly telling us something about the shape of the sequence $b_i$, we can also use these inequalities to derive lower bounds for $b_i$. Let us do so for the first few ones.

\me
For $n=4$ we note that if $q_2>1$ then $c^{(2)}_1\ge 0$ gives
$$
b_2 \ge   2d-5 . \\
$$
If $q_2>2$, then solving the quadratic equation $c_2^{(2)}\ge 0$, as in \cite{LP}-3.4, we deduce the stronger inequality
\begin{equation}\label{eqb2}
b_2\ge 2d  +\frac{\sqrt{8d-31}}{2}  -\frac{11}{2} .
\end{equation}
We note that expression under the square root is always positive, since $d\ge 5$ by the irreducibility assumption on $D$. We have an equality in (\ref{eqb2}) when $c_2^{(2)}=0$, which by Theorem \ref{thmMain2}-(b) is guaranteed when $\beta_3=1$. This is the case of a generic arrangement with $d=5$.

\me
For $n=5$ we note that $c_1^{(2)}\ge 0$ gives
$$
b_2\ge 2d-5 .
$$
If $q_2>2$, then solving the quadratic equation $c_2^{(2)}\ge 0$ we deduce the stronger inequality
$$
b_2\ge 2d  +\frac{\sqrt{8d-31}}{2}  -\frac{11}{2} .
$$
Note that these are the same inequalities as in $n=4$ case. The inequality $c_1^{(3)}\ge 0$ gives, if $q_3>1$,
$$
b_3>2b_2-3d+7 .
$$
Solving for the quadratic equation $c_2^{(3)}\ge 0$ gives as in \cite{LP}, if $q_3>2$, the stronger inequality
\begin{equation}\label{eqb3}
b_3\ge \frac{13}{2} -3d +2b_2+\frac{\sqrt{73-24d+8b_2}}{2} .
\end{equation}
This is available as long as the expression under the square root is nonnegative, that is if $b_2\ge 3d-9$. We note that we get an equality in (\ref{eqb3}) if $c_2^{(3)}= 0$, which by Theorem \ref{thmMain2}-(b) is guaranteed when $\beta_4=1$. This is the case of the generic arrangement with $d=6$.

\begin{rem} From the above computations we note the following trend: the inequality $c_{i+1}^{(j)}\ge 0$ is better  than $c_i^{(j)}\ge 0$.  \end{rem}

\begin{ex} Let $D$ be a generic central hyperplane arrangement of degree $d$ in $\bC^{n}$, with $d>n$. Then
\begin{align*}
b_i & =\binom{d-1}{i}, \quad\beta_i =\binom{d-2}{i}, \quad c_1^{(i)} =\binom{d-3}{i},
\end{align*}
if $0\le i\le n-1$, otherwise all three numbers are zero. Also, $c_2^{(i)}=0$ when $d=j+3$ and $n=j+2$.
\end{ex}

\section{Remarks on  resonance varieties}

\subsection{Explicit equations for $R^1$.}\label{subsecComp}
There are well-known explicit bases for the vector spaces $H^i(U)$ such as the ``no broken circuits" sets, \cite{OT}. Hence the matrices representing $\phi_i$, and thus by Theorem \ref{propRefRes}, the equations for the resonance varieties $R^{\,i}_j(U)$ are very explicit. 

\begin{example}
Let us spell out how to obtain $R^1$ for the case when the hyperplane arrangement $D$ is the cone over a planar line arrangement $\{H_1,\ldots , H_{d}\}$ with at most triple points. Each line $H_j$ with $j\ne d$ defines a basis element of $H^1(U)$, and of its dual, i.e. a generator $x_j$ of the symmetric algebra $S=\bC[x_1,\ldots ,x_{d-1}]$. For each intersection point $P$ of two lines from $\{H_1,\ldots , H_{d-1}\}$, define 
$$i_P:=\max \{ i<d\ |\ P\in H_i\}.$$
Consider the pairs $(i,i_P)$ with $i<i_P$ and $P\in H_i$. These pairs index a basis of $H^2(U)$. The $b_2\times b_1$ matrix $M=(M_{(i,i_P),j})$ representing the map $\phi_1$ has entries
$$
M_{(i,i_P),j}=\left\{
\begin{array}{ll}
0 &\quad\text{ if }j\not\in\{i,i_P\}, \\
-x_{i_P} &\quad\text{ if }j=i,\\
x_i &\quad\text{ if }j=i_P,
\end{array}\right.
$$
if $P$ is a multiplicity-two point of $\{H_1,\ldots ,H_{d-1}\}$, and
$$
M_{(i,i_P),j}=\left\{
\begin{array}{ll}
0 &\quad\text{ if }j\not\in\{i,k,i_P\}, \\
x_i &\quad\text{ if }j=i_P,\\
-x_{i_P}-x_k &\quad\text{ if }j=i,\\
-x_{i_P}+x_i &\quad\text{ if }j=k,
\end{array}\right.
$$
if $P$ is a multiplicity-three point of $\{H_1,\ldots ,H_{d-1}\}$ and $\{H_i,H_k,H_{i_P}\}$ are the three lines passing through $P$. The equations of $R^1$ in $S$ are the $(d-2)$-minors of the matrix $M$.
\end{example}

\me
Coming back to the general case, note that $R^1(U)$ is the common zero locus of  $\binom{b_1}{\beta_1}\cdot\binom{b_2}{\beta_1}$ polynomials, the $\beta_1$-minors of $\phi_1$. Next proposition states that only $\binom{b_2}{\beta_1}$ of these polynomials are necessary. To our knowledge this is currently the best reduction of the computational complexity of $R^1(U)$.

\begin{prop}\label{propFactor} Let $1\le i\le b_1$. Then $R^1(U)$ is the common zero locus of $x_i^{-1}m$, where $m$ ranges over the maximal minors of the $b_2\times \beta_1$-matrix obtained from $\phi_1$ by removing the $i$-th column.  
\end{prop}
\begin{proof}
By \cite{BE}-Theorem 3.1, the map $\bigwedge^{\beta_1}\phi_1^*$ factors through the map $\phi_0$. The map $\phi_0$ is represented by the matrix $[x_1\ldots x_{b_1}]^t$. Let $A$ be the $1\times \binom{b_2}{\beta_1}$-matrix such that $[x_1\ldots x_{b_1}]^t\cdot A$ represents the map $\bigwedge^{\beta_1}\phi_1^*$. Then the entries of $A$ are the elements $x_i^{-1}m$, with $m$ as above. By {\it loc. cit.}, part (b), the common vanishing locus of the entries of $A$ is the same as the support of the Fitting ideal $\cI(\phi_1)$. By Corollary \ref{corFitt}, the support of $\cI(\phi_1)$ is $R^1(U)$.
\end{proof}

A similar reduction of the number of necessary polynomials to define the higher resonance varieties $R^i(U)$ is available following \cite{BE}.

\subsection{Reformulation of Question \ref{queCod}.}\label{subsRefQ} Let $I=I(\phi_{n-2})\subset S$ be the homogeneous ideal of $R^{n-2}$ from Corollary \ref{corFitt}. That is $I$ is the ideal generated by the $\beta_{n-2}$-minors of $\phi_{n-2}$. Let $q=q_{n-2}$ be the codimension of $I$. The following gives an algebraic reformulation of Question \ref{queCod}.

\begin{prop}
$q>1$ iff  for every linear form $f$ and homogeneous element $m\in S$, $fm\in I$ implies $m\in I$.  
\end{prop}
\begin{proof}
We have that $q>1$ iff every associated prime ideal $P$ of $I$ has codimension $>1$. Since $I$ is homogeneous, every associated prime ideals of $I$ is also homogeneous. By definition, $P$ is an associated prime ideal of $I$ if it annihilates a nonzero element $m+I$ of $S/I$. This is equivalent to annihilating a homogeneous element, \cite{E}-3.12. Since $S$ is a polynomial ring, a homogeneous prime ideal $P$ has codimension 1 iff $P$ is generated by a linear form $f$.
\end{proof}

We note that Question \ref{queCod} has a positive answer if $D$ is itself the generic central hyperplane section of an indecomposable central essential hyperplane arrangement in a higher number of variables, by Lemma \ref{lemResTrunc} and the first inequality of Theorem \ref{propCod},

\subsection{Structure of $R^{\,i}_j$.} It is known that the refined resonance varieties $R^{\,i}_j(U)$ are supported on finitely many linear subspaces: \cite{CO,CS} using Hodge theory, \cite{L-first} using  deformations, \cite{LY} using linear algebra. Let us mention a few facts that follow from Theorem \ref{propRefRes}. 

\me
 The refined resonance varieties $R^{\,i}_j$ are finite intersections of varieties swept by vector spaces of particular type. More precisely, for a $p\times q$ matrix of linear forms with $p\le q$, the support of the locus defined by the vanishing of the $p$-minors is swept by the linear spaces determined by the simultaneous vanishing of the coefficients of a linear combination of rows (or of columns), \cite{E}-Exercise A2.19.

\me
On a different note, let us see next what it means to prove, using only linear algebra, that $R^{1}_j$ are supported on linear subspaces. Fix a basis $e_1,\ldots ,e_p$ of $\bC^p$. Let $M(x)$ be a matrix of size $q\times p$, $p\le q$, with entries linear in $\bC[x_1,\ldots ,x_{p}]=\bC[x]$, such that $M(\bfa)\bfb=-M(\bfa)\bfb$ for any complex vectors $\bfa$ and $\bfb$ of size $p$, and $\rank (M(\bfa))=p-1$ for generic vectors $\bfa$.  Define $R^{1}_j$ to be the common zero locus of the $(p-j)$-minors of $M(x)$. For example, if $M(x)$ is the matrix representing $\phi_1$ in (\ref{eqF}),  then $R^1_j$ is the refined resonance variety by Theorem \ref{propRefRes}. 

\begin{lem}
Let $\bfa\in R^1_j-R^1_{j+1}$. Let $N$ be a submatrix of $M(\bfa)$ of size $(p-j-1)\times p$ that contains a $(p-j-1)\times (p-j-1)$ submatrix $N_0$ with nonzero determinant. Let $\bfb_1,\ldots, \bfb_{j+1}$ be the vectors obtained as the $(p-j)$-minors, containing $N_0$, of the $(p-j)\times p$ matrix obtained from concatenating vertically the matrix $(e_1 \ldots e_p)$ with $N$. Then $\ker M(\bfa)$ is a dimension $j+1$ linear subspace with basis $\bfb_1,\ldots, \bfb_{j+1}$.
\end{lem}
\begin{proof}
We prove first that  $\bfb_1,\ldots, \bfb_{j+1}$ are nonzero and linearly independent. Let $k\in\{1,\ldots j+1\}$. Since the minor corresponding to $N_0$ is nonzero, there exists a coordinate $e_{i_k}$ with nonzero entry in $\bfb_k$. The set $\{i_{k'}\mid k'\in\{1,\ldots j+1\}\}$ corresponds to the columns of $N$ not contributing to the minor $N_0$. Hence the $i_{k'}$-th coordinate of $\bfb_k$, with $k'\ne k$, is zero. This shows that $\bfb_1,\ldots, \bfb_{j+1}$ are linearly independent.

 Since $\bfa\in R^1_j-R^1_{j+1}$, the codimension of $\ker M(\bfa)$ is $p-j-1$. Thus, we only need to show now that $M(\bfa)\bfb_k=0$ for all $k\in\{1,\ldots j+1\}$. Consider the $i$-th entry $(M(\bfa)\bfb_k)_{i}$ of $M(\bfa)\bfb_k$. This entry is, by the definition of $\bfb_k$, either a $(p-j)$-minor of $M(\bfa)$, or the determinant of the vertical concatenation of a $(p-j-1)\times (p-j)$ submatrix  of $N$ containing $N_0$ with one of its rows. While in the second case $(M(\bfa)\bfb_k)_{i}$ vanishes because of  the repeated row, in the first case it vanishes because the rank of $M(\bfa)$ is $p-j-1$.
\end{proof}


Given the explicit description of generators of the spaces $\ker M(\bfa)$, one would like then to prove, via linear algebra, that there are only finitely many such $\ker M(\bfa)$ for $\bfa\in R^1_j-R^1_{j+1}$, and they form the support of $R^1_j-R^1_{j+1}$. This is done for hyperplane arrangements in \cite{LY}-Corollary 3.7. Outside this case, it is not clear how to characterize the class of matrices $M(x)$ having these properties.

\subsection{Other possible bounds.}
As in \cite{LP}, one can try to obtain bounds on Betti numbers of $U$ by displaying certain subspaces of $H^i(U)$ and counting their dimension. Note that $\Lam^i H^1(U)\ra H^i(U)$ is not injective on  decomposable forms, as it is known that there are monomials that  vanish in the Orlik-Solomon algebra. However, we can ask the following. For each $0\le i<n-1$ let $W_i$ be a vector subspace of $H^1(U)$ transversal to $R^{i}(U)$, such that $H^1(U)=W_0\supset \ldots \supset W_{n-2}$.  Define for $0<i<n$
\begin{align*}
\Sigma_i:  =\{ v_1\wedge \ldots \wedge v_i \in \Lam ^iH^1(U)\ | & \ \dim Span(v_1,\ldots ,v_i)=i\\ & \quad\text{ and } v_j\in W_{j-1}   \}.
\end{align*}

\begin{que}\label{queInj} Is the natural map $\Sigma_i \ra H^i(U)$  an injection ?
\end{que}

We note that this is true for $i=2$ and that a converse holds, in a certain sense,  see \cite{Fa} -2.13 and 3.1.
The numerical counterpart of this is:

\begin{prop}\label{propNum} If Question \ref{queInj} is true then, for $0<i<n$,
$$q_{0}+\ldots +q_{i-1} < b_i + \frac{i(i+1)}{2} .$$
\end{prop}
\begin{proof}
 Consider the Grassmanian $G(i,b_1)$ of dimension $i$ subspaces of $H^1(U)$, with  the Pl\"ucker embedding in $\bP$. Then $\bP(\Sigma_i)$ is the subvariety of  $G(i,b_1)$ consisting of subspaces  $L$ such that $\dim (L\cap W_j)\ge i-j$ for $0\le j\le i-1$. Noting that $\dim W_j=q_{j}$, the codimension of this Schubert variety in $G(i,b_1)$ can be computed by a standard formula, see for example \cite{GH}, and equals 
$$\codim \Sigma_i =(b_1-i)i +\frac{i(i+1)}{2}- (q_{0}+\ldots +q_{i-1}).$$
Since the dimension of $G(i,b_1)$ is $i(b_1-1)$, we have
$$
\dim \Sigma_i = (q_{0}+\ldots +q_{i-1}) - \frac{i(i+1)}{2} .
$$
The claim now follows from a positive answer to Question \ref{queInj}. 
\end{proof}

\newpage

\begin{center}
{\bf ERRATUM}
\end{center}

The proof of \cite[Theorem 1.1, p. 867\footnote{corresponding to p.11 in this arxiv version.}]{B} is incomplete: for the claimed equality of two sets, only one of the inclusions is proved. It was assumed wrongly that the other inclusion is easy. In fact, the other inclusion holds only under an additional hypothesis. We correct here Theorem 1.1 and Corollary 1.2 of \cite{B}. The rest of the results of \cite{B} stay the same, including the main results of the paper on deeper propagation and dimensional bounds for resonance varieties, and the positivity results on Betti numbers.  

\medskip
\noindent
{\bf Theorem 1.1.} Let $i\le n-2$. {\it (a) $\bP(R^i_1(U))$ is the support of the ideal $\cI_{\beta_i}(\phi_i)$. (b) $\bP(R^i_j(U))$ contains the support of the ideal $\cI_{\beta_i+1-j}(\phi_i)$, and equals it away from $\bP(R^{i-1}_1(U))$. (c) $\bP(R^1_j(U))$ is the support of the ideal $\cI_{\beta_1+1-j}(\phi_1)$. (d) $\bP(R^i_2(U))$ is the support of the ideal $\cI_{\beta_2+1-j}(\phi_2)$. (e) $\bP(R^i_j(U))$ is the support of the ideal $\cI_{\beta_i+1-j}(\phi_i)$ if $j\le 1+(n-3)/i$. (f) $\bP(R^i_j(U))$ is the support of the ideal $\cI_{b_i+1-j}(\phi_{i-1}\oplus \phi_i)$.}

\begin{proof} Part (a) is \cite[Proposition 3.4]{B}. The first claim of (b) is what is actually proven on p. 867 of \cite{B}. Indeed, it follows from: if $(R,P)$ is a local noetherian domain, and $M$ is an $R$-module minimally generated by $k$ or more elements, then $k-\rank(M)\le \dim\Tor _1^R(R/P,M)$. The proof of this statement involves the obvious fact that $\rank(N)\le k'$ where $N$ are the first syzygies of $M$ and $k'$ is the minimal number of generators of $N$ (note that there is a typo in {\it loc. cit.}). The converse of the statement holds if and only if $N$ is free. Since $N$ and $M$ are the stalks of $\coker(\phi_{i-1})$ and $\coker (\phi_i)$ in our case, by part (a) this gives the second claim in (b). Since $\bP(R^0_1(U))=\emptyset$, (c) follows. \cite[Corollary 1.3]{B} holds as stated due to the right inclusion in (b) and it follows from the same statements about the corresponding determinantal ideals. Thus $\bP(R^{i-1}_1(U))$ is a subset of the support of $\cI_{\beta_2+1-j}(\phi_2)$, and (d) follows by (b). Similarly, (e) follows from the second part of \cite[Corollary 1.3]{B} for determinantal ideals. Part (f) is well-known.
\end{proof}

We do not know if the original statement of Theorem 1.1, that $\bP(R^i_j(U))$ is the support of $\cI_{\beta_i+1-j}(\phi_i)$, holds in general. Due to the right inclusion of part (b) in Theorem 1.1 above, this change affects in \cite{B} only the statement of Corollary 1.2: 

\medskip
\noindent 
{\bf Corollary 1.2.} {\it Let $i\le n-2$. If $(\beta_{i-1}+j)(\beta_{i+1}+j)<d-2$, then $\bP(R^i_j(U))$ is connected whenever it is the support of $\cI_{\beta_i+1-j}(\phi_i)$, or more generally, $\bP(R^i_j(U))$ is connected away from the components of $\bP(R^{i-1}_1(U))$ which are disconnected from the support of  $\cI_{\beta_i+1-j}(\phi_i)$.}

\medskip

 I would like to thank Botong Wang for very useful discussions.

\end{document}